\documentclass[12pt]{amsart}
\usepackage[all]{xy}
\usepackage{amssymb}
\usepackage{amsthm}
\usepackage{hyperref}
\hypersetup{colorlinks=true,linkcolor=blue,citecolor=magenta}
\usepackage{amsmath}
\usepackage{amscd,enumitem}
\usepackage{verbatim}
\usepackage{eurosym}
\usepackage{float}
\usepackage{color}
\usepackage{dcolumn}
\usepackage[mathscr]{eucal}
\usepackage[all]{xy}
\usepackage{hyperref}
\usepackage{bbm}
\usepackage[textheight=8.5in, textwidth=6.4in]{geometry}
\newtheorem*{thm*}{Theorem}
\newtheorem*{conj*}{Conjecture}

\newtheorem*{remark}{Remark}

\newtheorem{theorem}{Theorem}[section]

\newtheorem{lemma}[theorem]{Lemma}

\newtheorem{proposition}[theorem]{Proposition}

\newtheorem*{example}{Example}
\newtheorem{corollary}[theorem]{Corollary}

\newcommand{\CL}{\mathrm{CL}}
\newcommand{\Z}{\mathbb{Z}}
\newcommand{\Q}{\mathbb{Q}}

\newcommand{\R}{\mathbb{R}}
\newcommand{\SL}{\operatorname{SL}}

\newcommand{\tor}{\mathrm{tor}}

\newcommand{\leg}[2]{\genfrac{(}{)}{}{}{#1}{#2}}

\numberwithin{equation}{section}

\begin{document}
\title[Heights of points on elliptic curves over $\Q$]{Heights of points on elliptic curves over $\Q$}
\author{Michael Griffin, Ken Ono and Wei-Lun Tsai}
\address{Department of Mathematics, 275 TMCB, Brigham Young University, Provo, UT 84602}
\email{mjgriffin@math.byu.edu}
\address{Department of Mathematics, University of Virginia, Charlottesville, VA 22904}
\email{ken.ono691@virginia.edu}
\email{tsaiwlun@gmail.com}

\begin{abstract}   
In this note we obtain effective lower bounds for the canonical heights of non-torsion points
on $E(\Q)$ by making use of suitable elliptic curve ideal class pairings
$$\Psi_{E,-D}: \ E(\Q)\times E_{-D}(\Q)\mapsto \CL(-D).$$
In terms of the class number $H(-D)$  and $T_E(-D)$, a logarithmic function in $D$, 
we prove 
$$
\widehat{h}(P)> \frac{|E_{\tor}(\Q)|^2}{\left( H(-D)+ |E_{\tor}(\Q)|\right)^2}\cdot T_E(-D).
$$

\end{abstract}

\maketitle
\section{Introduction and statement of results}\label{Intro}
 
Throughout, suppose that $E/\Q$ is an elliptic curve given by
\begin{equation}\label{EModel}
E: \ \ y^2 = x^3 +a_4x +a_6,
\end{equation}
where $a_4, a_6\in \Z$, with $j$-invariant $j(E)$ and discriminant $\Delta(E).$ Furthermore, suppose that $E(\Q)$ has positive Mordell-Weil rank.
Each  point $P\in E(\Q)$ has the form $P=(\frac{A}{C^2},\frac{B}{C^3})$, with $A,B,C$ integers such that $\gcd(A,C)=\gcd(B,C)=1$. The {\it naive height} of $P$ is
$H(P)=H(x):=\max(|A|, |C^2|).$ 
The {\it logarithmic height} (or Weil height) is 
$h_W(P)= h_W(x):=\log H(P),$
and the {\it canonical height} is 
\begin{equation}
\widehat h(P):= \tfrac{1}{2}\lim_{n\to \infty}\frac{h_W(nP)}{n^2}.
\end{equation}

A conjecture of Lang  (see page 92 of \cite{Lang}) asserts that the canonical height $\widehat{h}(P)$ of every non-torsion point
$P\in E(\Q)$ satisfies
\begin{equation}\label{LangLowerBound}
\widehat{h}(P)> C_1 \log(|\Delta(E)|) + C_2,
\end{equation}
where $C_1>0$ and $C_2$ are absolute constants.
In a deep paper, Silverman \cite{Silverman81} proved Lang's general conjecture for every elliptic curve with (algebraic) integral $j$-invariant. Hindry and Silverman \cite{HindrySilverman} later proved the conjecture assuming the truth of the $abc$ Conjecture. In fact, their result applies for elliptic curves with
bounded Szpiro quotient.

We revisit  the more modest problem of deriving explicit lower bounds for such heights.
To place this challenging problem in context, we recall earlier work on this question.
In 1980, Anderson and Masser proved \cite{AM} that
\begin{equation}\label{one}
\widehat{h}(P)>\frac{\gamma_E}{\log(3)^6},
\end{equation}
where $\gamma_E$ is an effectively computable positive constant that depends on
estimates for the Weierstrass $\wp$-function and $\sigma$-function, and
inequalities from transcendental number theory. Further work of David and Petsche \cite{David, Petsche} give further lower bounds of a related nature.
More recently, Autissier, Hindry, and Pazuki \cite{AHP}, proved (in the case of  $K=\Q$) that
\begin{equation}\label{two}
\widehat{h}(P)> {\color{black} \frac{c\cdot |E_{\tor}(\Q)|^2}{h \log(3)^2}\cdot \log(3h)^{\frac{4}{3}}},
\end{equation}
where $c>0$ is an absolute constant, and $h:=\max\{1, h_W(j(E))\}.$ Here $h_W(\cdot)$ denotes the Weil  height of an algebraic number. These results, which are generally orders of magnitude smaller than the minimal heights of non-torsion points,
illustrate the difficulty of the task.\footnote{There are situations (e.g. when $E$ has rank $\geq 4$) where
their lower bounds are much stronger.}

We apply elliptic curve ideal class pairings (see (\ref{pairing}) and (\ref{formula})), maps of the form
$$
\Psi_{E,-D} \ : \ E(\Q)\times E_{-D}(\Q) \ \ \longrightarrow\ \  \CL(-D),
$$
where $\CL(-D)$ denotes the $\SL_2(\Z)$ equivalence classes of the discriminant $-D<0$ positive definite binary quadratic forms, and $E_{-D}$ is the $-D$ quadratic twist of $E.$ 
The lower bounds we obtain are of a completely different type, and are explicit and straightforward to calculate.
These results rely on upper bounds, over all $P\in E(\Q),$ for $\tfrac{1}{2}h_W(P)-\widehat{h}(P).$
To this end, we define the non-negative real number\footnote{The value 0 is attained by the identity in $E(\Q).$}
\begin{equation}
\delta(E):=\max_{P\in E(\Q)} \left(\frac{1}{2}h_W(P)-\widehat{h}(P)\right),
\end{equation}
and the logarithmic function
\begin{equation}\label{Tconstant}
T_E(-D,t):=\log\left(\frac{D}{4(t+1)^2} \right)-4\delta(E).
\end{equation}

\begin{theorem}\label{MainTheorem}
Suppose that $E(\Q)$ has positive rank. If $P\in E(\Q)$ is a point of infinite order, and $t$ is a positive integer for which 
$D=D(t):=4(t^3+a_4t-a_6)$ satisfies
\[ 
4(t+1)^2 \exp(4\delta(E))<D<(t+1)^2t^2,
\]
then
\[\widehat{h}(P) \geq \frac{|E_\tor(\Q)|^2}{(H(-D)+|E_\tor(\Q)|)^2} \cdot T_E(-D,t).
\]
\end{theorem}

\begin{remark}
The $T_E(-D)$ in the abstract is any choice of $T_E(-D,t)$ in the theorem above.
\end{remark}

\begin{example} We offer a typical example of Theorem~\ref{MainTheorem} with the elliptic curve
$$
E/\Q: \ \ y^2= x^3 -3x-4.
$$
Its Mordell-Weil group $E(\Q)$ has rank 1 and no non-trivial torsion. Its generator
$P_1=\left(8,22\right)$ has 
$\widehat{h}(P_1)\approx 1.107.$
 Proposition~\ref{deltaE} will show that $\delta(E)=0.$  We may choose $D(3)=D=88,$ and use the fact that $H(-88)=2$ to obtain\footnote{In the published version this was incorrectly given as $0.043.$} $\widehat{h}(P)\geq 0.035$ for all non-torsion points $P\in E(\Q)$.
 As with existing results,  we see that Theorem~\ref{MainTheorem} generally gives
 bounds which are far from optimal. We note that this curve has the special property
 that all of its Tamagawa numbers equal 1. Such a curve $E$ often has (see Proposition~\ref{deltaE}) the property that, for all $P\in E(\Q),$ that
 $$
 \widehat{h}(P)\geq \frac{1}{2}h_W(P).
 $$
\end{example}

Using the trivial upper bound for class numbers, we obtain the following corollary.

 \begin{corollary}\label{MainCorollary} 
 If $E(\Q)$ has positive rank
 and $t$ is any positive integer for which 
 $$
 4(t+1)^2\exp(4\delta(E))< D(t)<(t+1)^2t^2,
 $$
   then for every non-torsion point $P\in E(\Q)$ we have
$$
\widehat{h}(P) >  
\frac{\pi^2|E_\tor(\Q)|^2}{(\sqrt{D(t)}(\log D(t)+2)+\pi|E_\tor(\Q)|)^2} \cdot T_E(-D(t),t).
$$
\end{corollary}



To be useful, we need  explicit upper bounds for $\delta(E)$ which do not require the precalculation
of $E(\Q)$. Thanks to a
well-known theorem of Silverman (see Theorem 1.1 of \cite{Silverman}), we have such a bound.
If $P\in E(\Q)$, then
\begin{equation}\label{Sbound}
0 \leq \delta(E)\leq \tfrac{1}{8}h_W(j(E))+\tfrac{1}{12}h_W(\Delta(E))+0.973.
\end{equation}
Using this bound for $\delta(E)$ in Theorem~\ref{MainTheorem} and Corollary~\ref{MainCorollary}, one obtains explicit uniform results.
As suggested by the lower bounds in Corollary~\ref{MainCorollary}, one may choose the smallest positive integer $t$ which satisfies the double inequality for $D(t)$.
We conclude by noting that there are  natural situations
which guarantee that $\delta(E)=0,$ thereby relaxing the conditions we require for $t.$ In such cases, we generally obtain a sizable improvement to these uniform lower bounds. 

To make this precise, we elaborate and expand on an example of Buhler, Gross and Zagier \cite{BGZ}, which is based on a method of Tate
(for example, see \cite{Silverman88}).
Using (\ref{EModel}), we define the function
\begin{equation}
F_E(x):=\frac{1}{2}\log| x|_{\infty}+\frac{1}{8}\sum_{n=0}^{\infty}\frac{\log |z_n|_{\infty}}{4^n},
\end{equation}
where $|\cdot|_\infty$ is the usual archimedean valuation of $\R,$ and where for $n\geq 0$ we let
\begin{align}\label{sequence}
  x_0=x,\quad  z_n=1-\frac{2a_4}{x_n^2}-\frac{8a_6}{x_n^3}+\frac{a_4^2}{x_n^4},\quad \mathrm{and} \quad x_{n+1}=\frac{x_n^{4}-2a_4x_n^2-8a_6x_n+a_4^2}{4(x_n^3+a_4x_n+a_6)}.
\end{align}
If $x=x(P),$ the $x$-coordinate of $P\in E(\Q),$ then these expressions imply that $x_{n}=x(2^{n}P)$. 
Under suitable circumstances, we find that 
$$F_E(x_n)-\log|x_n|_{\infty}/2
$$
is rapidly convergent, as $n\rightarrow +\infty,$ allowing us to compute $\widehat{h}(P)$, and under special circumstances guarantee the vanishing of $\delta(E).$

For any finite place $v$,  we let $E_0(\mathbb{Q}_v)$ be the open subgroup of $E(\mathbb{Q}_v)$ consisting of points whose reduction is non-singular. Since $E(\mathbb{Q}_v)$ is compact, the quotient group $E(\mathbb{Q}_v)/E_0(\mathbb{Q}_v)$ is finite. The Tamagawa number $c_v:= |E(\mathbb{Q}_v)/E_0(\mathbb{Q}_v)|$ is the order of this quotient. In terms of Tamagawa numbers and the number of real connected components of $E$, we obtain the following criteria that guarantee the vanishing of $\delta(E)$.
 
 \begin{proposition}\label{deltaE} Suppose that $E/\Q$ has the property that $c_v=1$ for every finite place $v$  of bad reduction.  Then the following are true.
 \begin{enumerate}
 \item Suppose that $E(\R)$ has one connected component and that
  $$ a_4\leq 0,\qquad (\alpha,\infty)\subset\left\{x\in \R:2a_4x^2+8a_6x-a_4^2<0\right\},$$ where $\alpha$ is the real root of $x^3+a_4x+a_6$. Then for every $P\in E(\Q)$ we have
 $$\widehat{h}(P)\geq \tfrac{1}{2}h_W(P).$$

 \item Suppose that
 $E(\R)$ has two connected components, and that
  $$a_4\leq 0,\qquad (\gamma,\beta)\cup(\alpha,\infty)\subset\left\{x\in \R:2a_4x^2+8a_6x-a_4^2<0\right\},$$ where $\gamma<\beta<\alpha$ are the  real roots of $x^3+a_4x+a_6$.
 Then for every $P\in E(\Q)$ we have
 $$\widehat{h}(P)\geq \tfrac{1}{2}h_W(P).
 $$
 \end{enumerate}
\end{proposition}

\begin{remark} It is natural to ask how many elliptic curves satisfy the conditions in Proposition~\ref{deltaE}.  In \cite{GOT4} the authors
define the $L$-function
$$
L_{Tam}(s):=\sum_{m=1}^{\infty}\frac{P_{Tam}(m)}{m^s},
$$
where $P_{Tam}(m)$ denotes the proportion of elliptic curves over $\Q$ in short Weierstrass form with Tamagawa product $m$. 
Confirming a speculation of  Balakrishnan, Ho, Kaplan, Spicer, Stein, and Weigandt \cite{database},
Theorem~1.3 of \cite{GOT4} establishes that the {\it average value} of the Tamagawa products for elliptic curves over $\Q$ is $L_{Tam}(-1)=1.82\dots.$
 In the course of constructing this $L$-function, the authors prove that $P_{Tam}(1)=0.505\dots,$ which implies that over half of all elliptic curves over $\Q$ have $c_v=1$ for all places of bad reduction. Corollary~1.5 of \cite{GOT4} proves that over $16\%$ of elliptic curves
 are covered by one of the two cases in Proposition~\ref{deltaE}.
\end{remark}

\section*{Acknowledgements} \noindent  The second author thanks the NSF (DMS-1601306 and DMS-2002265) and
the  UVa Thomas Jefferson fund. The authors also thank Henri Darmon, Dorian Goldfeld, Peter Sarnak, and the referee for useful conversations and comments related to this note.

\section{Proof of Theorem~\ref{MainTheorem},  Corollary~\ref{MainCorollary}, and Proposition~\ref{deltaE}}

\subsection{Nuts and Bolts}

The class numbers $h(-D)$ of the imaginary quadratic field $\Q(\sqrt{-D})$ counts the equivalence classes of integral positive definite binary quadratic forms of fundamental discriminant $-D$.
The Hurwitz-Kronecker class number $H(-D)$ for (possibly non-fundamental) negative discriminants counts the number of discriminant $-D$ forms, where each class $C$ is counted with multiplicity $1/{\text {\rm Aut}}(C).$ Hurwitz class numbers satisfy a particularly nice multiplicative formula relative to class numbers with fundamental discriminant.  If $-D=-D_0 f^2,$ where $-D_0$ is a negative fundamental discriminant, then  (for example, see
Corollary 7.28 of \cite{Cox})
\begin{equation}\label{Hecke}
H(-D)=\frac{h(-D_0)}{\omega(-D_0)}\cdot \sum_{d\mid f}\mu(d)\chi_{-D_0}(d)\sigma_1(f/d),
\end{equation}
where $\mu(\cdot)$ is the M\"obius function, $\omega(-D_0)$ is half the number of units in $\Q(\sqrt{-D_0}),$ $\chi_{-D_0}(\cdot)$ is the corresponding Kronecker character, and $\sigma_1(n)$ is the sum of positive divisors of $n$.

The authors have previously \cite{GO1, GOT} obtained   lower bounds for $h(-D)$ and $H(-D)$ using
elliptic curve ideal class pairings
\begin{equation}\label{pairing}
\Psi_{E,-D}:\ E(\Q)\times E_{-D}(\Q) \ \ \longrightarrow\ \  \CL(-D),
\end{equation}
where $\CL(-D)$ is the discriminant $-D<0$ class group of binary quadratic forms and $E_{-D}$ is the $-D$ quadratic twist of $E$ with a suitable rational point.
Such pairing were introduced and previously studied by Buell, Call and Soleng \cite{Buell, BuellCall, Soleng}.
Instead of obtaining lower bounds for class numbers, here we combine these maps with upper bounds for class numbers
to derive information about the canonical height of points on $E(\Q)$.

To make this precise, for discriminants
$-D<0$ we let $E_{-D}$ be the quadratic twist (with nonstandard normalization)
$$
E_{-D}: \  -D\cdot \left (\frac{y}{2}\right )^2=\, x^3 +a_4x+a_6.
$$
Let $P=(\tfrac{A}{C^2},\tfrac{B}{C^3})\in E(\Q),$ with $A, B, C\in \Z$, and $Q=(u,v)\in E_{-D}(\Q),$
with $u,v\in \Z$. Moreover, suppose that
 $v\neq 0$, with $v$ even if $-D$ is odd.
If we let $\alpha:=|A-u C^2|$ and  $G:=\gcd(\alpha, v^2),$ then there are integers $\ell$
for which $F_{P,Q}(X,Y)$ defined below is a discriminant $-D$ positive definite integral binary quadratic form. 
\begin{equation}\label{formula}
{\color{black}
F_{P,Q}(X,Y)=\frac{\alpha}{G} \cdot X^2+\frac{2w^3 B+\ell  \cdot \tfrac{\alpha}{G}}{C^3v}\cdot XY +\frac{\left({ 2w^3B+
\ell \cdot \tfrac{\alpha}{G}}\right)^2+C^6 v^2{D}}{4C^6v^2\cdot \frac{\alpha}{G}}\cdot Y^2.
}
\end{equation}
If $P$ is the point at infinity, then we further define 
\[
F_{P,Q}(X,Y)= X^2+\ell\cdot XY +\frac{D+\ell^2}{4}\cdot Y^2.
\]
The following theorem was previously proved by two of the authors (see Theorem 2.1 of \cite{GO1}).

\begin{theorem}\label{ThmQF}{\text {\rm [Theorem 2.1 of \cite{GO1}]}}
Assuming the notation and hypotheses above, $F_{P,Q}(X,Y)$ is well defined (e.g. there is such an $\ell$)  {\color{black} in $\CL(-D)$}.
Moreover, if $(P_1,Q_1)$ and $(P_2, Q_2)$ are two such pairs for which $F_{P_1,Q_1}(X,Y)$ and $F_{P_2,Q_2}(X,Y)$ are $\SL_2(\Z)$-equivalent, then $\frac{\alpha_1}{G_1}=\frac{\alpha_2}{G_2}$ or $\frac{\alpha_1\alpha_2}{G_1G_2}\geq D/4$.
\end{theorem}

\subsection{Proof of Theorem~\ref{MainTheorem}}

To obtain Theorem~\ref{MainTheorem}, we apply this result to points of the form (if any) $Q=(-t,1)\in E_{-D}(\Q)$, where $t\in \Z^+.$ 
 It is equivalent to show that 
 \[
H(-D)\geq |E_\tor(\Q)|\cdot \left(\sqrt{\frac{ T_E(-D,t)}{\widehat h(P)}}-1\right).
 \]
We prove this inequality following the proofs of Theorem 4.1 and Proposition 3.3 of \cite{GO1}, but making use of several simplifications and improved bounds possible since we are considering only the rank $1$ lattice generated by the point $P$. We also note that the $\delta(E)$ used in \cite{GO1} is an approximation that uses Silverman's bound given in (\ref{Sbound}), plus $\tfrac12 \log 2,$ which accounts for the additional factor of $4$ appearing in the lower bound of $D$ in the statement of the theorem.

Let $\langle P\rangle$ denote the subgroup of $E(\Q)$ generated by $P$. Since $\widehat h(nP)=n^2\widehat h(P),$ we have that 
\[
\#\{P'\in\langle P\rangle  \ \mid \ \widehat h(P')\leq\tfrac14  T \} 
 \ = \ 2 \left\lfloor \sqrt{\frac{T}{4\widehat h(P)}}\right\rfloor+1.
\]
This implies that 
\begin{equation}\label{count1}
\#\{P'\in E(\Q)  \ \mid \ \widehat h(P')\leq \tfrac14 T\}  \ \geq \ |E_\tor(Q)|\cdot \left(2 \sqrt{\frac{T}{4\widehat h(P)}}-1\right).
\end{equation}

The point $Q=(-t,1)$ is such a point for the curve $E_{-D(t)}(\Q).$
If $P_1, P_2\in E(Q)$ satisfy $\widehat h(P_1),\widehat h(P_2)\leq \tfrac14 T_E(-D,t),$ then Theorem~\ref{ThmQF}  (following the proof of Theorem 4.1 of \cite{GO1}) implies that either  $F_1:=F_{P_1,Q}(X,Y)$ and $F_2:=F_{P_2,Q}(X,Y)$ are inequivalent forms, or $P_1=\pm P_2$. 
   However, we also note that we may take $F_{P,Q}(X,Y)=F_{-P,Q}(-X,Y)$. These quadratic forms are inverses under Gauss's composition law, and so they are inequivalent unless $P=-P$.

Using (\ref{count1}) with $T=T_E(-D,t)$,  we conclude that 
\[
H(-D)\geq  |E_\tor(\Q)|\cdot \left(\sqrt{\frac{ T_E(-D,t)}{\widehat h(P)}}-1\right).
 \]
This completes the proof.

\subsection{Proof of Corollary~\ref{MainCorollary}}
 
Corollary~\ref{MainCorollary} follows immediately from the classical
 upper bound for the Hurwitz-Kronecker class number. 

\begin{lemma}\label{LBound}
If $-D=-D_0f^2$ is a discriminant, where $-D_0<0$ is fundamental, then 
$$
H(-D)\leq\frac{\sqrt{D}(\log D+2)}{\pi}.
$$
\end{lemma}
\begin{proof} 
First, we recall the class number formula (for example, see Theorem 10.1 of \cite[p. 321]{Hua})
\begin{align}\label{ClassFormula}
h(-D_0)=\frac{\omega(-D_0)\sqrt{D_0}}{\pi}\cdot L(1,\chi_{-D_0}),
\end{align}
where $L(s,\chi_{-D_0})$ is the Dirichlet $L$-function for the quadratic character $\chi_{-D_0}(\cdot)=\leg{-D_0}{\cdot}$.
Using (\ref{Hecke}) and (\ref{ClassFormula}), we obtain 
\begin{align*}
    H(-D)=\frac{\sqrt{D_0}}{\pi}\cdot L(1,\chi_{-D_0})\cdot \sum_{d\mid f}\mu(d)\chi_{-D_0}(d)\sigma_1(f/d).
    \end{align*}
Since the summand is a multiplicative function, it turns out that
\begin{align}\label{Hbound}
    H(-D)\leq\frac{\sqrt{D}}{\pi}\cdot L(1,\chi_{-D_0})\cdot\prod_{\substack{p|f\\p:\mathrm{prime}}}\left(1-\frac{\chi_{-D_0}(p)}{p}\right)&\leq \frac{\sqrt{D}}{\pi}\cdot L(1,\psi_{-D}),
\end{align}
where $\psi_{-D}$ is a quadratic character modulo $D$.
Furthermore, using Abel summation (for example, see Theorem 14.3 of \cite[p. 330]{Hua}) we obtain
\begin{align}\label{Lbound}
    L(1,\psi_{-D})\leq \left|\sum_{k=1}^{D}\frac{\psi_{-D}(k)}{k}\right|+\left|\sum_{k\geq D+1}\frac{\psi_{-D}(k)}{k}\right|\leq 1+\int_{1}^{D}\frac{dx}{x} + \frac{D}{D+1}\leq \log D +2.
\end{align}
Combining (\ref{Hbound}) and (\ref{Lbound}), we obtain the desired conclusion
$$
H(-D)\leq \frac{\sqrt{D}(\log D+2)}{\pi}.
$$
\end{proof}

  \subsection{Proof of Proposition~\ref{deltaE}}
The two claims are an elaboration and generalization of an example of Buhler, Gross and Zagier \cite{BGZ}.  
For brevity, we prove the first case, as a similar argument proves the second case. 
Let $\mathcal{O}\neq P\in E(\Q)$ be a rational point. We use the fact that the canonical height $\widehat{h}(P)$ can be computed by the sum of local heights; namely, 
\begin{align*}
    \widehat{h}(P)=\sum_{v\in{M_{\Q}}}\widehat{h}_{v}(P),
\end{align*}
where $M_{\Q}$ contains all the places of $\Q$. 

Since $c_v$ is trivial for every place of bad reduction, we use the Theorem 5.2 b) of \cite{Silverman88} (choosing $N=1$ and $n=0$)   to conclude that the places of bad reduction make no contribution to the height.
Furthermore, thanks to the calculation in \cite[p. 475]{BGZ}, we have 
\begin{align}\label{keyid}
    \widehat{h}(P)=\widehat{h}_{\infty}(P)=\frac{1}{2}h_W(P)+F_{E}(x(P))-\frac{1}{2}\textrm{max}(0,\log|x(P)|_{\infty}).
\end{align}
Finally, the assumptions that both
$$ a_4\leq 0 \ \ \ \ {\text {\rm and}}\ \ \ \  (\alpha,\infty)\subset\left\{x\in \R:2a_4x^2+8a_6x-a_4^2<0\right\}$$
guarantees that $x(P)\geq 1$ hold for all non-torsion points $P.$ Therefore, $\log|x(P)|_{\infty}\geq 0$ and 
$$F_{E}(x(P))-\frac{1}{2}\log|x(P)|_{\infty}\geq0.$$
Hence, by (\ref{keyid}), the first case follows immediately.

\end{document}